\documentclass[12pt]{amsart}

\usepackage{amsmath}
\usepackage{amssymb}
\usepackage{amsthm}

\usepackage[dvips]{graphics}
\usepackage{graphicx}
\usepackage{epsfig}
\usepackage{ifthen}
\usepackage[dvips]{color}  
\usepackage{amsmath,amsfonts,amsthm}
\usepackage{subeqnarray}
\usepackage{stmaryrd}
\usepackage{dsfont}

\newcommand{\supp}{\operatorname{supp}}

\renewcommand{\div}{\operatorname{div}}

\newcommand{\Rr}{{\mathbb{R}}}

\newcommand{\Zz}{{\mathbb{Z}}}

\newcommand{\Tt}{{\mathbb{T}}}

\newcommand{\Hh}{{\overline{H}}}

\newcommand{\cqd}{\hfill $\square$}

\newcommand{\Pf}{{\noindent \sc Proof. \ }}
\newcommand{\Rm}{{\noindent \sc Remark. \ }}

\newtheorem{teo}{Theorem}
\newtheorem{lem}{Lemma}
\newtheorem{pro}{Proposition}

\newtheorem{defi}{Definition}

\theoremstyle{definition}

\begin{document}
\title[W-measures and the semi-classical limit to A-M measure]
{  Wigner measures and  the  semi-classical limit to the
Aubry-Mather measure}

\author{Diogo A. Gomes (*), Artur O. Lopes (**), and Joana Mohr (***)}

\maketitle

\begin{abstract}

In this paper
we investigate the asymptotic behavior of the semi-classical limit of  Wigner measures defined on the tangent bundle of the one-dimensional torus.  In particular we show the convergence of Wigner measures to the  Mather measure on the tangent bundle, for energy levels above the minimum of the effective Hamiltonian.

The Wigner measures $\mu_h$  we consider are associated to $\psi_h,$ a distinguished critical solution  of the Evans' quantum
action given by $\psi_h=a_h\,e^{i\frac{u_h}h}$, with $a_h(x)=e^{\frac{v^*_h(x)-v_h(x)}{2h}}$, $u_h(x)=P\cdot x+\frac{v^*_h(x)+v_h(x)}{2},$ and
$v_h,v^*_h$ satisfying the equations
\begin{equation*}\begin{split}
- \, \frac{h\, \Delta v_h}{2}+ \frac{1}{2} \, | P + D v_h \,|^2 + V &= \overline{H}_h(P),\\
  \, \frac{h\, \Delta v_h^*}{2}+ \frac{1}{2} \, | P + D v_h^* \,|^2 + V &= \overline{H}_h(P),
\end{split}\end{equation*}
where the constant $\overline{H}_h(P)$ is the $h$ effective potential and $x$ is on the torus.
L.\ C.\ Evans considered limit measures $|\psi_h|^2$ in $\mathbb{T}^n$, when $h\to 0$, for any $n\geq 1$.

We consider the limit measures on the phase space $\mathbb{T}^n\times\mathbb{R}^n$, for $n=1$, and, in addition, we obtain rigorous asymptotic expansions for the functions $v_h$, and $v^*_h$, when $h\to 0$.
\end{abstract}

\thanks{
\begin{center}

(*) Partially supported by CAMGSD/IST through FCT Program POCTI -
FEDER and by grants PTDC/MAT/114397/2009,
UTAustin/MAT/0057/2008, PTDC/EEA-ACR/67020/2006, PTDC/MAT/69635/2006,
and PTDC/MAT/72840/2006, and by the bilateral agreement Brazil-Portugal (CAPES-FCT) 248/09.\\(**) Partially supported by CNPq, PRONEX -- Sistemas
Din\^amicos,  INCT, and beneficiary of CAPES financial support.\\ (***)
Partially supported by a CNPq postdoc scholarship.
\end{center}
}

\section{Introduction}

Consider the mechanical system defined by a Lagrangian of the form
$L(x,v)=\frac12 |v|^2 \, - V(x) - P v,$ with
$x$ in the $n$-dimensional torus $\mathbb{T}^n= \frac{\mathbb{R}^n}{\mathbb{Z}^n},$ $v\in\mathbb{R}^n,$
and fixed $P\in \mathbb{R}^n,$ and the associated Hamiltonian given by $H(p,x)=
\frac12 (P+ p)^2 \, + V(x).$ Here, $V$ stands for the
potential energy, which we assume to be symmetric in order to simplify
our arguments.

Before we present our results, let us review some basic facts of the Aubry-Mather theory.
Let $\mathcal M$ denote the set of probability measures on the Borel $\sigma$-algebra of
$\mathbb{T}^n\times\mathbb{R}^n.$

 The Mather problem  on $\mathbb{T}^n$ (see \cite{Mat},   \cite{CI} and    \cite{Fa}) consists in determining
the probability measures $\mu\in \mathcal M$ which minimize the action
\begin{equation} \label{Ma} \int_{\mathbb{T}^n\times\mathbb{R}^n}\,\bigg[\, \frac12 |v|^2 \, - V(x)\,\bigg]\, \,d\mu(x,v),\end{equation}
among the probabilities $\mu \in \mathcal M$ such that (the holonomic condition)
$$ \int v\, .\,D \phi\, d \mu =0,\,\,\, \forall \phi\in C^1 ( \mathbb{T}^n) ,$$
and,
$$ \int \, v \, d \mu=V,$$
for a fixed vector $V$.

For dimension $n>1,$ these minimizing measures on the tangent bundle, called Mather measures, do not need to be unique and are supported on sets which are not attractors for the flow. They are invariant by the correspondent  Euler-Lagrange flow.

It is known that the Mather measures on $\mathbb{T}^n\times\mathbb{R}^n$ are supported on graphs (the projection on $\mathbb{T}^n$ is injective). The probabilities we get when we project on $\mathbb{T}^n$ are called  projected Mather measures. They are probabilities on $\mathbb{T}^n$. In most of the cases the projected probability is not absolutely continuous with respect to the Lebesgue measure on the torus.

The Mather measure described above can be alternatively obtained as a minimization of the action of E-L invariant  measures without constrains considering a new Lagrangian of the form (see \cite{CI})
$$L(x,v)=\tfrac12 |v|^2 \, - V(x) - P v.$$

The result presented by L.\ C.\ Evans in \cite{Ev1} considers the projected Mather measure, but his methods do not extend to the analysis which we will consider here on the tangent bundle.

An example of Mather measure on the tangent bundle can be shown on figure 1. In this case consider the only E-L invariant measure with support on the top curve in figure 1. The periodic curve which is inside the open set determined by the two separatrices  does not carry a Mather measure.

In the above case the projected Mather measure can be computed by action angle variables (see
section 7.3 page 110 (7.24) in \cite{PR}, or, section 7.2 page 159 (7.22) in \cite {OA}).

Now consider the Hamiltonian $H$ associated to $L$, that is given by $$ H(p,x)=\sup_v[pv-L(x,v)],$$  and its associated Hamilton-Jacobi equation
$$H(\nabla\phi(x),x)=c,$$
where $c$ is a constant.
In general, we cannot find $C^1$ solutions $\phi$ of this equation and, therefore, we consider viscosity solutions, as follows.

We say that a function $\phi:U\to\mathbb{R}$ is a viscosity sub-solution in $U$ of the equation $H(\nabla\phi(x),x)=c$
if, given any $C^1 $ function $\eta:U\to\mathbb{R}$, and any $x_0\in U,$ we have
$$H(\nabla\eta(x_0),x_0)\leq c,$$
whenever $\phi-\eta$ attains a maximum value at $x_0.$
Dually, we say that a function $\phi:U\to\mathbb{R}$ is a viscosity super-solution in $U$ of the same Hamilton-Jacobi
equation if, given any $C^1 $ function $\eta:U\to\mathbb{R}$, and any $x_0\in U,$ we have
$H(\nabla\eta(x_0),x_0)\geq c,$ whenever $\phi-\eta$ attains a minimum value at $x_0.$
Finally, a function is a backward viscosity solution of the Hamilton-Jacobi equation if it is both a viscosity sub and super-solution (See \cite{Ev1}, \cite{Fa}, \cite{Ev2}, \cite{BG}, and \cite{A1}).

Now let us define a forward viscosity solution of $H(\nabla\phi(x),x)=c$, see \cite{Fa}. To do that we need to define
 $$ \check{H}(p,x)=\sup_v[-pv-L(x,v)].$$
 Let $\check{\phi}$ be a backward viscosity solution of $\check H(\nabla\check\phi(x),x)=c,$ then we say that $\phi^*=-\check\phi$ is a forward viscosity solution of $H(\nabla\phi(x),x)=c$.

 Note that  $\check{H}(p,x)= H(-p,x).$

L.C. Evans has shown that  for a Lagrangian of the form
  $L(x,v)=\frac{ |v|^2}2 \, - V(x) - P \,v  $,   the solutions $v_h$ and $v_h^*$ of \eqref{b1} and
\eqref{b2} converge (maybe in subsequence), respectively,  to backward and forward  viscosity solutions, $\phi$ and $\phi^*$, of $H(\nabla\phi(x),x)=\Hh (P)$, where $\Hh (P)$ is sometimes called the Ma\~{n}\'{e} critical value, see \cite{Ev1}.

If $\phi:U\to\mathbb{R}$ is a viscosity solution of $H(\nabla\phi(x),x)=\Hh (P)$ then the equality $H(\nabla\phi(x),x)=\Hh (P)$ holds,
in the classical sense, at every point $x$ of differentiability of $\phi.$ (See Corollary 4.4.13 in \cite{Fa}.)

In the case the Mather measure is unique, the viscosity solution is unique, up to an additive constant
(See Sections 7 and 8 in \cite{Fa}, Section 4.9 \cite{CI}, \cite{Ev1}, \cite{BG},
and \cite{Gom3}), and the solution $v_h$ of \eqref{b1} does converge to the backward viscosity solution of $H(\nabla\phi(x),x)=\Hh (P)$,
as $h\to 0.$

There exist a value $P_{crit}$ such that for $P>P_{crit}$ there exists viscosity solutions (see \cite{Fa}) on the corresponding level of energy $\Hh(P)$. For the case of a potential with just a maximum in the one-dimensional  torus the constant energy lines
are described in figure 1. The periodic curve in the top corresponds to $P$ above the critical value $P_{crit}$. The periodic curve which is inside the open set determined by the two separatrices  not.

We will consider here only the one-dimensional (that is, $n=1$) case for a certain general class of potentials for which the Mather measure is unique, and the projected Mather measure is absolutely continuous with respect to the invariant measure. In this one-dimensional setting the Mather measure is the unique  measure on a certain energy level $\Hh(P)$ (depending on $P$) which is invariant by the Euler-Lagrange flow.

For level of energy $\Hh(P)$ which projects on the all one-dimensional  torus one can consider the
function $p^{+} (x)$ which graph determines the curve on the tangent bundle.  The integral
$\int_{-1/2}^{1/2} p^{+} (x) dx =P$, relates the energy level $\Hh(P)$ to $P$. The function $\Hh(P)$ is differentiable on $P$ \cite{EG1}.

We point out that our result does cover a large class of interesting cases, such as the one-dimensional Lagrangians
$L(x,v)=\frac12 |v|^2 \, - V(x) + w_x(v),$ with $w$ a nontrivial closed form, for which the projected Mather measure is absolutely continuous.

Evans introduced a quantum minimization problem for his quantum action which is similar to the classical
minimization problem of Mather's theory. He considered the semi-classical problem of limit measures on the configuration space $\mathbb{T}^n,$ showing that in the semi-classical limit with $h \to 0,$  when the probability associated to the wave $\psi_{h}$ converges, then the limit measure is a projected Mather measure \cite{Ev1}.

To describe our main result we need to recall some facts from \cite{Ev1}.
We analise the Wigner measure associated to the critical function
$$\psi_h(x) = a(x)\,e^{i \frac{P x+ \tilde{v}(x)}{h}}$$
of the Evans's quantum action
\[
S[\psi]=\int_{\Tt^n} \tfrac12 h^2 |D a|^2 +
\left[\tfrac12 |P+D_x\tilde{v}|^2-V(x)\right] a^2 dx,
\]
under the constraints
\[
\int_{\Tt^n} a^2 dx=1,
\]
stationary current
\[
\div(a^2 (P+D_x\tilde{v}))=0,
\]
and total current intensity
\[
\int_{\Tt^n} a^2 (P+D_x\tilde{v})dx= Q.
\]

The function $a$ above take only real values.

The analysis of the properties of the critical solutions described  above  is the so called Evan's quantum action minimization problem.

The Evans' critical solution (see \cite{Ev1}) is given by the periodic function $\psi_h=a_h\,e^{i\frac{u_h}h}$, where
$a_h(x)=e^{\frac{v^*_h(x)-v_h(x)}{2h}}$, and $u_h(x)=P\cdot x+\frac{v^*_h(x)+v_h(x)}{2}.$ Here,
$v_h,v^*_h$ satisfy the equations
\begin{equation}\label{b1}
 - \, \frac{h\, \Delta v_h}{2}+ \frac{1}{2} \, | P + D v_h \,|^2 + V= \overline{H}_h(P),
\end{equation}
\begin{equation}\label{b2}
  \, \frac{h\, \Delta v_h^*}{2}+ \frac{1}{2} \, | P + D v_h^* \,|^2 + V= \overline{H}_h(P),
\end{equation}
where the constant $\overline{H}_h(P)$ is called the effective Hamiltonian.

We also assume that $v_h^*-v_h$ is normalized, that is,
\begin{equation}\label{norm}\int_{\Tt^n} e^{\frac{v_h^*(x)-v_h(x)}h  }dx=1,\end{equation}
and, therefore,
$$\int_{\Tt^n} |\psi_h (x)|^2 dx=1.$$

   In the present case we consider   the limit function $v_0$ which is a viscosity solution (also almost everywhere differentiable) and satisfies div $(( P + v_{0}') \, \sigma_0) =0$,
where $\sigma_0$ is the projected Mather measure  (see \cite{Ev1} section 3, \cite{BG} or \cite{EG1}). We will comment more about this fact in the last section.

The function $\psi_h$ is not periodic.

N. Anantharaman \cite{A1}, \cite{A3}  and L.C. Evans \cite{Ev1}
had previously (and in arbitrary dimension) studied the convergence on the configuration space of the probability
measures $|\psi_h (x)|^2(x)\, dx$.

One of the main motivations in the study of the critical solution given by $\psi_h=a_h\,e^{i\frac{u_h}h}$, with
$a_h(x)=e^{\frac{v^*_h(x)-v_h(x)}{2h}}$, $u_h(x)=P\cdot x+\frac{v^*_h(x)+v_h(x)}{2},$ as
$h \to 0,$ is its relation to the viscosity solutions of the Hamilton-Jacobi equation which was described before.

C.  Evans \cite{Ev1} considered probabilities "on the $n$-dimensional torus".
We are interested here in the semi-classical limit of the Wigner measure (to be defined soon) associated to $\psi_h$, when $h \to 0$, and its relation to the Aubry-Mather measure "on the tangent bundle" of the one-dimensional torus. (See (\cite{Mat}, \cite{CI}, \cite{Fa}, and \cite{BG}.) This means that we consider the problem on the phase space $\mathbb{T}^n\times\mathbb{R}^n.$

We point out that the quantum minimal action problem is closely related to the Bohm's viewpoint of Quantum Mechanics, which is connected to the classical Hamilton-Jacobi equation with the addition of an extra diffusion term. (See Section 19 in \cite{Jo}, \cite{Re} and \cite{Vo}.) Observe also that we are considering an action which is quite different from the action used by F.\ Guerra and L.\ Morato in \cite{GM}.

Before we proceed, we discuss some background material on the semi-classical analysis and Wigner measures on the torus
$\Tt^n=\mathbb{R}^n / \mathbb{Z}^n ,$ which we identify with $[-\frac{1}2,\frac{1}2)^n.$

Given a state $\psi\in L^2(\Tt^n)$, one would like to compute the averages of observables, such as momentum or energy.  Here, observables are linear operators $b$ and their averages are given by $\langle \psi, b\psi\rangle.$

Frequently, quantum observables are pseudo-differential operators associated to a smooth symbol $a(x, p).$
In the periodic setting, the momenta $p$ are quantized and, instead of taking values in $\Rr^n,$
they belong to the lattice $h \Zz^ n+ \frac{P}{2\pi} .$
In the pseudo-differential calculus, there is not a unique way to associate an operator to a given symbol.
A common choice is to associate to a symbol $a$ the Weyl operator $a^W$ given by the Weyl quantization rule
\[
\langle \psi, a^W \psi\rangle
=
\sum_{p\in h \Zz^n+ \frac{P}{2\pi}}
\int_{\Tt^n}
\int_{\Tt^n}
a(x, 2\pi p) \bar \psi(x+\tfrac12 y)\psi(x-\tfrac12 y) e^{-\frac{2 \pi i p y}h} dy dx,
\]
which has the advantage of associating self-adjoint operators to real symbols.
(See \cite{dGo}, and \cite{A3} for related results.)
Notice the multiplicative factor $2\,\pi$ of the momentum variable, due to the product coordinates of the torus.
In our one-dimensional setting,
\[
\langle \psi, a^W \psi\rangle =
\sum_{p\in h \Zz+ \frac{P}{2\pi}} \int_{-\frac{1}2}^{\frac{1}2} \,\bigg[\int_{-\frac{1}2}^{\frac{1}2} a(x, 2\pi p) \bar
\psi(x+\tfrac12 y)\psi(x-\tfrac12 y) e^{-\frac{2 \pi i p y}h} dy\,\bigg] dx.
\]

\begin{figure}\label{Fig1.1}
\begin{center}
\epsfig{file=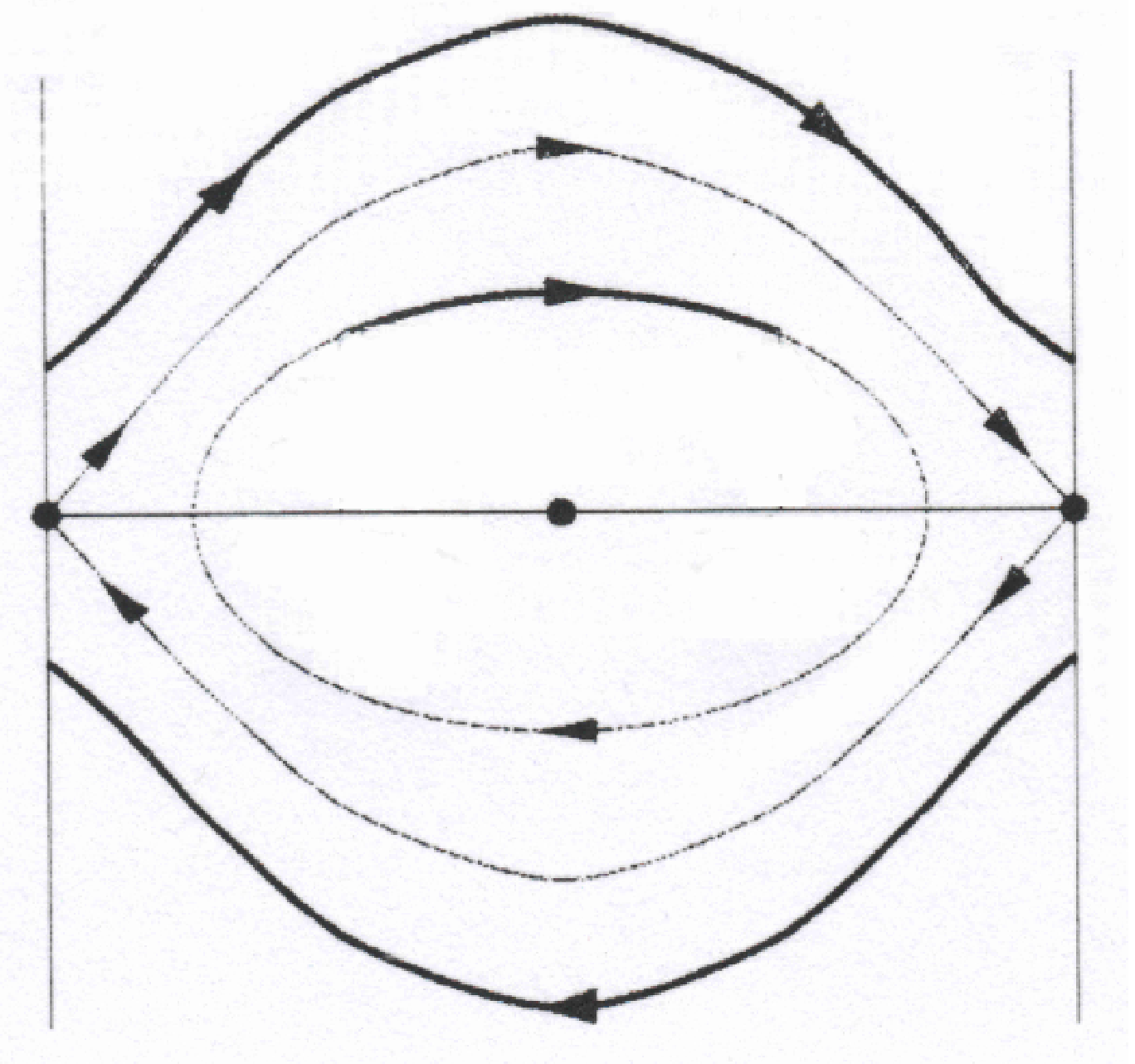, width=8cm}
\end{center}
\end{figure}

If $a$ is simply a function of $x,$ then
\[
\langle \psi, a^W \psi\rangle=
\int_{\Tt^n} a(x) |\psi|^2 dx.
\]
and if $a$ depends only on $p,$ the average value is
\[
\langle \psi, a^W \psi\rangle=
\sum_{p\in {h\Zz^n}+ \frac{P}{2\pi}}
a(2\pi p) |\hat \psi(p)|^2,
\]
where
\[
\hat \psi(p)=\int_{\Tt^n} \psi(x) e^{-\frac{2\pi i p \cdot x}h}dx
\]
are the Fourier coefficients.

Given any $\psi$ such that $\int |\psi|^2 dx =1,$ we define  $W_h^\psi$ on $\Tt^n\times (h\,\Zz^n + \frac{P}{2\pi})$ by
\[
W_h^\psi (x,p)=
\int_{\Tt^n}
\bar \psi(x+\tfrac12 y)\psi(x-\tfrac12 y) e^{-\frac{2 \pi i p y}h} dy.
\]
It follows that
\[
\langle \psi, a^W \psi\rangle
=
\sum_{p\in h \Zz^n+ \frac{P}{2\pi}}
\int_{\Tt^n}
a(x, 2\pi p) W_h^\psi(x,p)dx.
\]

The values of $p$ where chosen in the set $(h\,\Zz^n + \frac{P}{2\pi})$ because $\psi_h$ is not periodic

Note that   $< \psi, a^W \psi>$ gives real values when  applied to smooth real functions with compact support (See \cite{Zh}).
The value corresponding to a  nonnegative real function $a$ it is not necessarily nonnegative.

It is well known that if $\int_{\Tt^n} |\psi|^2 dx=1$ for some $\psi\in L^2(\Tt^n),$ then
\[
\sum_{p\in h \Zz^n+ \frac{P}{2\pi}}
\int_{\Tt^n}W_h^\psi (x,p)\,dx=1.
\]

For fixed $\psi$ it is known that limit points, when $h \to 0$, of Wigner distributions are positive measures \cite{Zh}.

We denote by   $\mu_h^{\psi_h}$ the Wigner distribution associated to $\psi_h$ on $\Tt^n\times\mathbb{R}^n$ given by
\begin{equation}\label{wigner} \sum_{p\in h \Zz^ n+ \frac{P}{2\pi}}
\int_{\Tt^n}
a(x, 2\pi p) W_h^\psi(x,p)dx        =\int a(x,p) \,d  \mu_h^{\psi_h}.
\end{equation}


We want to consider for each value $h$ the Wigner distribution $\mu_h =\mu_h^{\psi_h}$ associated to the Evans' critical $\psi_h$. Note that the normalization we considered before for $v_h^*-v_h$ does not create indeterminacy for $W_h^{\psi_h}$.

A very complete  description of what is known when $h\to 0$, in problems related
to the projected Aubry-Mather measures is \cite{A1}, \cite{A3}.

In \cite{AIPS} the case where do exist several different Mather measures is analyzed and
the question of selection of the limit projected Mather measure is considered.

In \cite{Ev1} the problem is  described in a different setting and in terms of the
quantum action. This result is for the torus of dimension $n$.

In the general case the Mather measures are singular with respect to Lebesgue measure.

On the geodesic case for negative curvature manifolds these measures are supported on geodesic laminations (appendix \cite{A2}) and a Large Deviation Principle is known under the hypothesis of uniqueness of the maximizing probability. For nonnegative curvature the Large Deviation problem around a closed geodesic with only points of curvature zero is analyzed in \cite{LR}.

Our results are all in the one-dimensional setting, and our main result is the following.

\begin{teo}\label{teo1}
Let $L:\mathbb{T}\times\mathbb{R}\to\mathbb{R}$ be a Lagrangian of the form $$L(x,v)=\tfrac12 |v|^2 \, - V(x) - P v,$$
with fixed $P> P_{crit},$ and a smooth potential energy $V$ symmetric with respect to the origin 0 that has a unique
non-degenerate minimum point at the origin.

 If $\mu_{h}$ denotes the Wigner distribution associated to the Evans' critical $\psi_h$, then $\mu_h$ converges to the Mather measure on the tangent bundle. In this case, the Mather measure is the unique invariant measure for the Euler-Lagrange flow on a certain energy level (which depends on $P$).

Moreover, consider a smooth function $f$ supported on a compact set $A\subset \mathbb{T}\times\mathbb{R}.$
\begin{enumerate}
  \item If the support of $f$  does intersect the energy level $\overline H(P),$ then
$$\int f(x,p)\,d\mu_{h}(x,p)\to 2\pi \sum^{2}_{i=1} \int\frac{f(x_i( p),2\pi{p})\;p^+(\bar x)} {|(p^{+})'(x_i( p))|\;p^+(x_i( p))}dp, $$
for certain functions $p^+$ and $x_i(p)$ with $p^+(x_i( p))=2\pi p.$

The value $ \bar x$ appears in a natural way and is responsible for the normalization which assures us that we get  a probability in the limit.
The right hand side is just the canonical expression of the invariant density for the Hamiltonian flow via action-angle variables.

  \item If the support of $f$ does not intersect the energy level $\overline H(P),$ then
$$\int f(x,p)\,d\mu_{h}(x,p)\to 0.$$
\end{enumerate}

In both cases case we are able to estimate the speed of convergence.

Our measure convergence claim is in  the following sense: $\mu_h$ converges to $\mu$ if $\int fd\mu_h\to \int f d\mu,$ for
each smooth functions $f$ with the property that its support does not intersect two certain lines $p=p^P_{min}$ and $p=p^P_{max}$ to be defined in the following section.

\end{teo}

The symmetry hypothesis on $V$ simplifies our computations, but our method can be adapted to more
general situations (taking up much more pages). For instance, just below \eqref{a2} this symmetry hypothesis results in the same property for the viscosity solution, and this makes the analysis of the problem much more simple.

In the proof of our theorem we obtain an estimate of the velocity of the convergence of
$\int f(x,p)\, d\mu_h (x,p)$ for a  certain class of test function $f(x,p)$ with compact support, where $\mu_h$ is the Wigner distribution associated to the family $\psi_h.$ In addition, we also establish the asymptotic behavior of the solutions of \eqref{b1} and
\eqref{b2} when $h\to 0.$ This requires a rigorous asymptotic expansion analysis which, to the best of our knowledge, cannot
be found in the existing literature.

It is true that in this one-dimensional case we could avoid the mention of Mather measures, and consider only invariant measures on constant energy levels but since we are going to use the setting and the results of \cite{Ev1}, which are presented in the context of the Aubry-Mather theory, we believe that our discussion so far could be helpful.

Note that if one knew in advance that the limit
of the Wigner distribution is an invariant measure for the hamiltonian flow, the result would be trivial,
because we are in a situation where there is a unique invariant measure on each energy
level. However,  $h $ is only a $O(h)$-quasimode of the quantum hamiltonian (this was proved
by C. Evans \cite{Ev1}) and we do not know a priori that the limits of the Wigner measures are invariant.

The proof relies on the
facts that the phase space picture is completely explicit, that the Aubry-Mather measure
is unique, and that the solution to the Hamilton-Jacobi equation is unique, smooth, and
explicit; thus, this approach is specific to dimension 1.

The semiclassical limit of Wigner measures that minimize the quantum
action was considered in \cite{GV}. This  was inspired by the
previous work \cite{Ev1} which only uses classical tools. We now
revisit this problem and compute the rates of convergence (above the
critical level of energy) for one-dimensional Hamiltonians.
We will consider  $P> P_{crit}$, where $P_{crit}= \inf \{ P: \overline{H}(P)> \min \overline{H}\}>0.$

Concerning possible extensions of our results to higher dimensions, we acknowledge that the control of the stationary phase method in the first part of our work, although being much more difficult, could possibly be obtained. Anyway, in the last part of our work (see Section 3) we have to control an asymptotic expansion, for which we need a Mather measure that projects on the entire torus. We strongly believe that this control is not obtainable with our methods without that assumption. An extension in this direction will require a completely different approach.


This paper is organized as follows. In the next section
we  derive the correct asymptotic behavior of the integrals over
Wigner measures when $h\to 0$. In the last section we describe the
asymptotic expansion of the function $v_h$ in terms of $h,$
which we will need for the proof of Theorem \ref{teo1}.

\section{ The asymptotic limit of critical Wigner measures}

The  critical Wigner measure is the Wigner measure associated to a critical solution $\psi_h$. For simplicity we will omit the word critical from now on.
We want to investigate the asymptotic limit of Wigner measures, when $h\to 0$.

Remember that $W_h^{\psi_h}(x,p)=\int_{\Tt} \bar \psi_h(x+\frac{y}2) \psi_h(x-\frac{y}2) e^{-\frac{2\pi i p y}{h}}dy$, hence for any continuous $f$ with support on a compact set  $A\subset \Tt \times \mathbb{R}$
$$\int f(x,p)\, d\mu_h (x,p)\,=\,\sum_{\hat{p}\in h \Zz+ \frac{P}{2\pi}}\int_{\Tt}\int_{\Tt}f(x,2 \pi\hat{p}) e^{\frac{v^*_h(x+\frac{y}{2})-v_h(x+\frac{y}{2})+v^*_h(x-\frac{y}{2})-v_h(x-\frac{y}{2})}{2h}} \cdot
 $$
$$\cdot e^{i[P\cdot \frac{y}{h}+\frac{v^*_h(x+\frac{y}{2})+v_h(x+\frac{y}{2})  -v^*_h(x-\frac{y}2)-v_h(x-\frac{y}{2})}{2 h}  ]}e^{-\frac{2\pi i \hat{p} y}{h}} \;dx dy. $$

Note that, as $\hat{p}\in\{ h m+ \frac{P}{2\pi}$; $m\in \Zz$\}, we can write

$$\int f(x,p)\, d\mu_h (x,p)\,=\,\sum_{\hat{p}\in h \Zz+ \frac{P}{2\pi}}\int_{\Tt}\int_{\Tt}f(x,2 \pi\hat{p}) e^{\frac{v^*_h(x+\frac{y}{2})-v_h(x+\frac{y}{2})+v^*_h(x-\frac{y}{2})-v_h(x-\frac{y}{2})}{2h}} \cdot
 $$
$$\cdot e^{i[\frac{v^*_h(x+\frac{y}{2})+v_h(x+\frac{y}{2})  -v^*_h(x-\frac{y}2)-v_h(x-\frac{y}{2})}{2 h}  ]}e^{-2\pi i m y} \;dx dy. $$

Note that the integrand is periodic on the variable $y$

We point out that  the  analysis of the case of integrals of functions depending only on $x$, namely,
$f(x,p)=f(x)$, is well  known. It follows from the results in \cite{A1}
and \cite{Ev1}. In particular, the projection of the above Wigner
measures converges to the projected Mather measure (on the torus) \cite{Ev1}
and large deviations are also well understood (in the case the
Mather probability is unique) according to  Proposition 3.11 \cite{A1}.

 In the one-dimensional case
(above the $x$-axis)   the level
of energy consist exactly of one periodic trajectory when $P>P_{crit}$. Indeed, note that if the energy is high the modulus of the velocities on this level of energy is high, and, therefore,  by conservation of energy the level of energy has to be a curve like the one  on the top of figure 1.
So, it follows from the graph property \cite{CI}
\cite{Fa} \cite{BG} \cite{Gom3} that  the Mather  probability has
support in this periodic trajectory.

The case of $P\in (-P_{crit}, P_{crit})$ (where $P$ attains the minimum value of  $\overline{H}$) requires a different analysis and will not be
considered here.

We will show later that $W_h^{\psi_h}$
converges to the Mather measure which has support in such
trajectory.

For $P> P_{crit}$, let us find the backward viscosity solution of  $H( \phi' (x) ,x)= \overline H(P)$. If we denote by $p^+_P (x)=\sqrt{2(\overline H(P) -V(x))}>0$
then
$$\phi_P(x)=\phi(x) =  \,\int_{-\frac{1}2}^x \, p^+_P (s)ds - \,P\, \big(x+\frac{1}2\big)$$ is a solution, and hence a viscosity solution, of $H( \phi' (x) ,x)= \overline H(P)$.
It is easy to see that in this case the forward viscosity solution $\phi^*$ is such that  $\phi^*=\phi$.

The function $p^{+}_P$ is such that its graph determines a level of energy $\overline H(P)$(as for example the curve shown  on the top of figure 1).

Note that  $p^+_P (x)=P+\phi_P'(x) $. Then, we have  $\int_{-\frac{1}2}^{\frac{1}2} p^+_P (s)\, ds =P$.

 Let us denote by $p^P_{min}$ and $p^P_{max}$ the numbers such that $p^P_{min}\leq p^+_P (x)\leq  p^P_{max}$ for all $x\in \Tt$, note that $p^P_{min}>0$.

 The Mather measure with
rotation number $Q =  D_P \overline H(P)$ is defined by the unique periodic
trajectory supported in the graph $(x, p^+_P (x)).$

\begin{defi} Let $f$ and  $g$   be   functions of the variable $h>0$. We say that $f(h)=O(g(h))$ as $h\to 0$ if there exist a constant $K>0$ and $\delta>0$ such that $|f(h)|\leq K|g(h)|$ whenever $0<h<\delta$.

\end{defi}

\begin{defi} Let $f(h)$ and  $\psi(h)$   be two functions satisfying

 $\displaystyle \lim_{h\to 0}\frac{f(h)}{\psi(h)}=1$, then we say that $f$ is asymptotic to $\psi$, or $\psi$ is an asymptotic approximation to $f$, and, we write $f(h)\sim \psi(h)$, as $h\to 0$.

\end{defi}

In the last section we will study the asymptotic behavior of the functions $v_h$ and $v_h^*$, the solutions of \eqref{b1} and \eqref{b2}.

We prove that for any integer $\gamma\geq 2$ there are functions
functions  $v_i,v_i^*$, $i=0,1$, and $w_\gamma, w_\gamma^*$ such that the following  functions
\begin{equation}\label{eq3}\hat{v}^\gamma_h:= v_0  + h\, v_1 +  h^2 w_\gamma\,\,\,,\,\,\,
\hat{v}^{\gamma*}_h:= v^*_0  + h\, v^*_1 +  h^2 w_\gamma^*,\end{equation}
 approximate uniformly the functions $v_h, v_h^*$, to a order depending on $\gamma$.
More precisely, we can choose $w_\gamma$ and $w_{\gamma}$ given by a finite sum of the form
\[
w_\gamma=\sum_{j=2}^\gamma h^{j-2} v_j, \quad w_\gamma^*=\sum_{j=2}^\gamma h^{j-2} v_j^*,
\]
where all functions $v_j, v_j^*$ are smooth.
Then if we define $g_h=:v_h-\hat v_h^\gamma$ and $ g^*_h=:v^*_h-\hat v_h^{\gamma*}$, by proposition \ref{estimC0} of the last section, we have that
\begin{equation}\label{eq1}| g_h(x)-g_h(x_h)|=O(h^{\frac{\gamma}2}), \end{equation} and
\begin{equation}\label{eq2}| g^*_h(x)-g^*_h(x_h)|=
O(h^{\frac{\gamma}2}) .\end{equation}
Note that the error term is only $O(h^{\frac{\gamma}2})$ in the uniform norm, and we do not make any claim on boundedness of its derivatives.

In addition,  we will   also show that
\begin{equation}\label{eq4}\phi_P=\phi=v_0=v_0^* \,\,\,,\,\,\, v_1=-v_1^*.\end{equation}

\Rm By the normalization hypothesis \eqref{norm} for  $v_h$ and $v_h^*$,  for each $h$, there exists $x_h\in\Tt$ such that
$e^{\frac{v_h^*(x_h)-v_h(x_h)}h}=1$, then ${v_h^*(x_h)-v_h(x_h)}=0$. This implies that

\begin{equation}\label{normalization} g^*_h(x_h)-g_h(x_h)=\hat v^\gamma_h(x_h)-\hat v^{\gamma*}_h(x_h).  \end{equation}

Now we will use the notion of asymptotic approximation to simplify the expression of the Wigner measure.

   We define
$$ \tilde F_{P,\hat p}(h)=\tilde F(h):=\int_{\Tt} \int_{\Tt}f(x,2\pi \hat p)\, e^{\frac{v^*_h(x+\frac{y}2)-v_h(x+\frac{y}2)+
v^*_h(x-\frac{y}2)-v_h(x-\frac{y}2)}{2h}}\cdot
 $$
 $$\cdot e^{i[P\cdot \frac{y}{h}+\frac{v^*_h(x+\frac{y}2)+v_h(x+\frac{y}2)  -
v^*_h(x-\frac{y}2)-v_h(x-\frac{y}2)}{2 h} -\frac{2\pi
\hat{p} y}{h} ]} \;dx dy,$$
and
$$F_{P,\hat p}(h)=F(h):=\int _{\Tt} \int _{\Tt}f(x,2\pi \hat p)\,e^{[-v_1(x+\frac{y}2)-v_1(x-\frac{y}2)+2v_1(\bar x)]} e^{\frac{i}{h}S_{\hat p}(x,y)      }dx dy,$$
where $S_{\hat p}(x,y):=P\cdot y+v_0(x+\frac{y}2) - v_0(x-\frac{y}2)  -2\pi\hat{p} y,\,$ and $\displaystyle\,\bar x:=\lim_{h\to 0} x_h$.

We point out that the value $ \bar x$ appears in a natural way from the normalization hypothesis described by $x_h$ (see Remark above) and  assures us that we get  a probability in the limit. It has a multiplicative effect on the limit measure and therefore its influence does not depend on the subsequence we choose (see expression (1) in Theorem 1).

In the following proposition we use the results obtain in the last section, that is, the asymptotic expansion of $v_h$ and $v_h^*$.

\begin{pro}
Fix $\gamma \geq 3$. Then there exists a smooth function $\eta$
\begin{align*}
F(h)=&\tilde F(h)+h\int _{\Tt} \int _{\Tt} \eta(x,2\pi p, h)\,e^{[-v_1(x+\frac{y}2)-v_1(x-\frac{y}2)+2v_1(\bar x)]} e^{\frac{i}{h}S_{\hat p}(x,y)      }dx dy\\
&+O(h^{(\gamma-2)/2}).
\end{align*}
Furthermore, the function $\eta$ is uniformly bounded for all $h$, together with all its derivatives.
\end{pro}

\begin{proof}

First note that
\begin{align*}
&\quad\frac{v^*_h(x+\frac{y}2)-v_h(x+\frac{y}2)+
v^*_h(x-\frac{y}2)-v_h(x-\frac{y}2)}{2h}\\
&=\frac{g_h^*(x+\frac{y}2)-g_h(x+\frac{y}2)+g_h^*(x-\frac{y}2)-g_h(x-\frac{y}2)-2g^*_h(x_h)+2g_h(x_h)}{2h} \\
&+\frac{\hat v^{\gamma*}_h(x+\frac{y}2)-\hat v^\gamma_h(x+\frac{y}2)+
\hat v^{\gamma*}_h(x-\frac{y}2)-\hat v^\gamma_h(x-\frac{y}2)+2g^*_h(x_h)-2g_h(x_h)}{2h}.
\end{align*}
Using  equations \eqref{eq4} and \eqref{normalization},  we get
\begin{align*}
&\quad \frac{v^*_h(x+\frac{y}2)-v_h(x+\frac{y}2)+
v^*_h(x-\frac{y}2)-v_h(x-\frac{y}2)}{2h}
\\
&= \frac{g_h^*(x+\frac{y}2)-g_h(x+\frac{y}2)+g_h^*(x-\frac{y}2)-g_h(x-\frac{y}2)-2g^*_h(x_h)+2g_h(x_h)}{2h}
\\
&-v_1(x+\frac{y}2)-v_1(x-\frac{y}2)+ 2v_1(x_h)\\
&
+ \frac{h}{2}[ w_\gamma^*(x+\frac{y}2)+w_\gamma^*(x-\frac{y}2)-w_\gamma(x+\frac{y}2)
-w_\gamma(x-\frac{y}2)-2w_\gamma^*(x_h)+2w_\gamma(x_h)   ].
\end{align*}
Let
\begin{align*}
&\omega_1(h, x, y)\\&=
\frac{ w_\gamma^*(x+\frac{y}2)+w_\gamma^*(x-\frac{y}2)-w_\gamma(x+\frac{y}2)
-w_\gamma(x-\frac{y}2)-2w_\gamma^*(x_h)+2w_\gamma(x_h)   }{2}.
\end{align*}
Using equations \eqref{eq1} and \eqref{eq2}  we get
\begin{align}
\label{eter1}
&e^{\frac{v^*_h(x+\frac{y}2)-v_h(x+\frac{y}2)+
v^*_h(x-\frac{y}2)-v_h(x-\frac{y}2)}{2h}} \\\notag
&= e^{-v_1(x+\frac{y}2)-v_1(x-\frac{y}2)+2v_1(x_h)} e^{h \omega_1(h, x, y)}e^{ O(h^{\frac{\gamma-2}{2}})        }.\\\notag
&=e^{-v_1(x+\frac{y}2)-v_1(x-\frac{y}2)+2v_1(x_h)}\\\notag
&\quad +e^{-v_1(x+\frac{y}2)-v_1(x-\frac{y}2)+2v_1(x_h)} e^{h \omega_1(h, x, y)}-1)\\\notag
&\quad +e^{-v_1(x+\frac{y}2)-v_1(x-\frac{y}2)+2v_1(x_h)} e^{h \omega_1(h, x, y)} (e^{ O(h^{\frac{\gamma-2}{2}})        }-1).
\end{align}
We should note that the last term of the
previous equality is $O(h^{\frac{\gamma-2}{2}})$ in $L^\infty$. The second term is $O(h)$
in the $C^k$ topology for any $k$.


By the same kind of reasoning we get
\begin{align}
&\frac{v^*_h(x+\frac{y}2)+v_h(x+\frac{y}2)  -
v^*_h(x-\frac{y}2)-v_h(x-\frac{y}2)}{2 h}\\\notag
&= \frac{v_0(x+\frac{y}2) - v_0(x-\frac{y}2)         }{h}+\\\notag
&\frac{h}2\bigg[w_\gamma^*(x+\frac{y}2)-w_\gamma^*(x-\frac{y}2)+w_\gamma(x+\frac{y}2)
-w_\gamma(x-\frac{y}2)\bigg]
\\\notag
&+\frac{g_h^*(x+\frac{y}2)-g^*_h(x_h)-g_h^*(x-\frac{y}2)+g^*_h(x_h)}{2h}\\\notag
&+\frac{g_h(x+\frac{y}2)-g_h(x_h)-g_h(x-\frac{y}2)+g_h(x_h)}{2h}.
\end{align}
Define
\[
\omega_2(x, y, h)=\frac{w_\gamma^*(x+\frac{y}2)-w_\gamma^*(x-\frac{y}2)+w_\gamma(x+\frac{y}2)
-w_\gamma(x-\frac{y}2)}{2}.
\]
Then
\begin{align}
\label{et2}
&e^{i[P\cdot \frac{y}{h}+\frac{v^*_h(x+\frac{y}2)+v_h(x+\frac{y}2)  -
v^*_h(x-\frac{y}2)-v_h(x-\frac{y}2)}{2 h} -\frac{2\pi
\hat{p} y}{h} ]}=e^{\frac{i}{h}S_{\hat p}(x,y)}e^{h \omega_2}e^{O(h^{\frac{\gamma-2}2})} \\\notag
&=  e^{\frac{i}{h}S_{\hat p}(x,y)}+e^{\frac{i}{h}S_{\hat p}(x,y)} (e^{h \omega_2}-1)+
e^{h \omega_2}e^{h \omega_2}(e^{O(h^{\frac{\gamma-2}2})}-1).
\end{align}
As before, the last term of this identity is $O(h^{\frac{\gamma-2}2})$ in $L^\infty$, the second term
is $O(h)$ in the $C^k$ topology for any $k$.

Therefore, by combining \eqref{eter1} and \eqref{et2}
we obtain the result.

\end{proof}

Let us denote by
$$\zeta(x,y):= e^{[-v_1(x+\frac{y}2)-v_1(x-\frac{y}2)+2v_1(\bar x)] }.$$

Hence $$I_f(h)=\sum_{\hat{p}\in h \Zz+ \frac{P}{2\pi}}\int_{\Tt}\int_{\Tt}\,  f(x,2\pi\hat{p})\; \zeta(x,y) \;e^{\frac{i}{h}S_{\hat p}(x,y)}\;dx dy.
 $$

We have to estimate the asymptotic of  the limit $\lim_{h\to 0}I_f(h)$.

In this paper we will use several asymptotic expansions for integrals,
namely the non-stationary and stationary phase methods. These
expansions are valid for smooth functions and are uniform depend on
bounds on a finite number of derivatives of the function involved.
Since in all cases the non-stationary or the stationary phase methods
are applied to (fixed) smooth functions all error estimates are
uniform.
In the use of oscillatory integrals it is possible to make expansions of  any order, if  the derivatives of order $k$  of the functions involved decay fast enough. In our case all functions are $C^\infty$, and then one can use the $C^\infty$ topology, or $C^k$ topology for any $k$ large enough.

 To estimate $I_f(h)$ we will use a two dimensional version of the
stationary phase method in the variables $(x,y)$. Let us fix $\hat{p}\in h \Zz + \frac{P}{2\pi}$, the main contribution in the integral  on $x$ and $y$ above  is due to the term $e^{\frac{i}{h}S_{\hat p}(x,y)}$. We have to find the points $(x,y)$ which are critical for $S_{\hat p}(x,y)$.
These are the points $(x,y)$ that are solutions of  the system

\begin{equation}\label{a1}\frac{\partial S_{\hat p}}{\partial y}(x,y)= \,P +\frac{1}2 v_0 ' (x+ \frac{y}2) + \frac{1}2 v_0 '(x-\frac{y}2) - \, 2 \, \pi \, \hat{p}\,=0,
\end{equation}
and
\begin{equation}\label{a2}\frac{\partial S_{\hat p}}{\partial x}(x,y)=\, v_0 '(x+ \frac{y}2) - v_0 ' (x-\frac{y}2)=0.
\end{equation}

Note that, by the definition of $v_0(x)=\phi(x)$, we have $v_0'(x)=p^+(x)-P$. If $x\neq 0$ in the equation \eqref{a2}, by the symmetry of the function $v_0$, the  only  solution  for $ v_0 '(x+ \frac{y}2) - v_0 ' (x-\frac{y}2)=0$ is $y=0$. And,  if $y\neq 0$ in equation \eqref{a2} we must have $x=0$.

 Now, if $y=0$, the   equation \eqref{a1} becomes
\begin{equation}\label{a3}p^+(x)-2\pi\hat p=0. \end{equation}
And, if $x=0$, by the symmetry of the function $p^+$,  the equation \eqref{a1} becomes \begin{equation}\label{a4}p^+(\frac{y}2)-2\pi\hat p=0. \end{equation}

For each $\hat p$ let us call $x_1 (\hat{p})$ and  $x_2 (\hat{p})$ the two points such that $p^+(x_i (\hat{p}))=2\pi \hat p,\;\; i=1,2.$

Remark: Note in figure 1 (for the top curve) that for any given $\hat{p}$ (in the image of  the projection in the $y$ coordinate of the points in the corresponding level of energy) there are two points $x_1,x_2$ such that
$(x_1,\hat{p}),$  $(x_2,\hat{p}),$ are on this level of energy. Due to the convexity of the term
$\frac{1}2\,|p|^2$ in the Hamiltonian $H(x,p)$ the level lines have this shape for the kind of symmetric potential we consider here.

Note that $p^{+} (0)$ is the supremum $p_{max}$ of the function $p^{+}$. Moreover, $p^{+} (-\frac{1}2)=p_{min}.$

For potentials $V$ with a larger number of points of maximum we could have a larger number of pre-images of $\hat{p}$ by $p^{+}$  but the reasoning would be basically the same.

Therefore, we have the following pair of possibilities for $(x,y)$ satisfying both \eqref{a1} and  \eqref{a2} :

 \begin{enumerate}\item $y=0$ and  $x=x_1 (\hat{p})$, or $x=x_2 (\hat{p})$, which are the  solutions of \eqref{a3},

 \item $x=0$ and  $y=2x_1 ( \hat{p}),$ or $y=2x_2 ( \hat{p}),$  which are the  solutions of \eqref{a4}.
\end{enumerate}

Summarizing, we have two possibilities:

(i) For  $\hat p \in h \Zz+ \frac{P}{2\pi} $ such that $2\pi \hat p\in [p^P_{min},p^P_{max}]$, the function $S_{\hat p}(x,y)$ has 4 critical points: $(x_i (\hat{p}),0)$ and  $(0,2x_i (\hat{p}))$, for $i=1,2$.

(ii) For  $\hat p \in h \Zz+ \frac{P}{2\pi} $ such that $2\pi \hat p\notin [p^P_{min},p^P_{max}]$, the function $S_{\hat p}(x,y)$ has no critical points.

In order to prove the theorem \ref{teo1} we need some auxiliary computations that we discuss now.  As the limit of the Wigner distribution is a positive measure (\cite{Zh} consider quite general hypothesis) it is enough to analyze the  following two cases:
\medskip

 Case I) If the support of $f$ is contained
  the strip $\Tt\times (p^P_{min}, p^P_{max})$.

 \medskip
 Case II) Suppose that the support of $f$ is contained in $\Tt\times [c, d]$, with $c,d$  such that
$[c, d] \cap [p^P_{min}, p^P_{max}]=\emptyset$.

\bigskip

Case I) First of all we  point out that when analyzing an estimate for the expression
$$I_f(h)=\sum_{\hat{p}\in h \Zz+ \frac{P}{2\pi}}\int_{\Tt}\int_{\Tt}\,  f(x,2\pi\hat{p})\; \zeta(x,y) \;e^{\frac{i}{h}S_{\hat p}(x,y)}\;dx dy,
 $$
we will have to consider for each $h$, in a separate argument, the contribution of
$$\int_{\Tt}\int_{\Tt}\,  f(x,2\pi\hat{p}_d)\; \zeta(x,y) \;e^{\frac{i}{h}S_{\hat p_d}(x,y)}\;dx dy,$$
where $\hat{p}_d= \hat{p}_d(h)$ is the first point in $h \Zz+ \frac{P}{2\pi}$ bellow $p_{max}^P$. We will estimate this term in the end of our reasoning.

We will denote

\begin{equation}\label{au2}
\sum_{<}\int_{\Tt}\int_{\Tt}\,  f(x,2\pi\hat{p})\; \zeta(x,y) \;e^{\frac{i}{h}S_{\hat p}(x,y)}dx dy,
\end{equation}
the sum is taken over all the $2\, \pi\,\hat{p} <\hat{p}_d$.

\medskip
Given $\epsilon>0$
consider $\eta^\epsilon$, a $C^\infty$ function such that,
$  \eta^\epsilon(y)=1$ for $|y|<\epsilon/2$, and,
$  \eta^\epsilon(y)=0$ for $|y|>\epsilon$.

We write
$$I_1^h (\epsilon)=\sum_{<}\int_{\Tt}\int_{\Tt}\, \eta^\epsilon (y)\, f(x,2\pi\hat{p})\; \zeta(x,y) \cdot
 $$
$$\cdot e^{i[P\cdot \frac{y}{h}+\frac{v_0(x+\frac{y}2)  -
v_0(x-\frac{y}2)}{ h}  -\frac{2\pi i \hat{p} y}{h}]}\;dx dy,
$$
and
$$I_2^h  (\epsilon)=\sum_{<}\int_{\Tt}\int_{\Tt} (1 -\eta^\epsilon (y) )\, f(x,2\pi\hat{p}) \;\zeta(x,y) \cdot
 $$
$$\cdot e^{i[P\cdot \frac{y}{h}+\frac{v_0(x+\frac{y}2)  -
v_0(x-\frac{y}2)}{ h}  ]}e^{-\frac{2\pi i \hat{p} y}{h}} \;dx dy.
$$

\begin{lem}\label{lema1} For any fixed $\epsilon$, we have
$\displaystyle \lim_{h\to 0} I_2^h(\epsilon) =0.$  More precisely we show that the expression goes like $O( h^\infty)$.

\end{lem}

\Pf

Consider a $C^\infty$ function $g(x, \hat{p})$ with compact support with all its derivatives with respect to $\hat{p}$ uniformly   bounded in $h$. Then, for fixed  $x$ and $y\neq 0$ and $A$ a bounded set we have
$$ \sum_{\hat{p}\in A,\, <} \, g(x,\hat{p} )e^{\frac{2 \pi \, i\,  \hat{p}\, y}{h
}}= \sum_{\hat{p}\in A,\, <} \, g(x,\hat{p} )  \frac{e^{\frac{2 \pi \, i\,  (\hat{p}+h)\, y}{h
}}-  e^{\frac{2 \pi \, i\,  \hat{p}\, y}{h
}}}{e^{2 \pi \, i\,  \, y} -1 } =$$
$$ \sum_{\hat{p}\in A,\, <} \,  \frac{ g( x,\hat{p} -h)- g( x, \hat{p})}
{e^{2 \pi \, i\,  \, y} -1 }\,e^{\frac{2 \pi \, i\,  \hat{p}\, y}{h
}}
 = \sum_{\hat{p}\in A,\, <} \,  \frac{ O(h)}
{e^{2 \pi \, i\,  \, y} -1 }\,e^{\frac{2 \pi \, i\,  \hat{p}\, y}{h
}}.
$$

We are using above the notation:   $\sum_{\hat{p}\in A,\, <}$ is the sum over the points in $A$ which are smaller then $\hat{p}_d$.

Note that in the above reasoning it is fundamental that  $y\neq 0$.

We can apply the above reasoning twice
$$ \sum_{\hat{p}\in A,\, <} \, g(x,\hat{p} )e^{\frac{2 \pi \, i\,  \hat{p}\, y}{h
}}=\sum_{\hat{p}\in A,\, <} \,  \frac{ g(x, \hat{p} -h)- g( x, \hat{p})}
{e^{2 \pi \, i\,  \, y} -1 }\,e^{\frac{2 \pi \, i\,  \hat{p}\, y}{h
}}
 =$$
 $$ \sum_{\hat{p}\in A,\, <} \,  \frac{ g(x, \hat{p} -h)- g( x, \hat{p})}
{e^{2 \pi \, i\,  \, y} -1 }\,\frac{e^{\frac{2 \pi \, i\,  (\hat{p}+h)\, y}{h
}}-  e^{\frac{2 \pi \, i\,  \hat{p}\, y}{h
}}}{e^{2 \pi \, i\,  \, y} -1 }
 =$$
$$ \sum_{\hat{p}\in A,\, <} \,  \frac{ g( x,\hat{p} -2h)-\,2\, g( x,\hat{p}-h )+ g (x,\hat{p} )}
{e^{2 \pi \, i\,  \, y} -1 }\,e^{\frac{2 \pi \, i\,  \hat{p}\, y}{h
}}$$
$$ =\sum_{\hat{p}\in A,\, <} \,  \frac{ O(h^2)}
{e^{2 \pi \, i\,  \, y} -1 }\,e^{\frac{2 \pi \, i\,  \hat{p}\, y}{h
}}.$$

Let $K\subset \mathbb{R}$ be a bounded set such that $f(x,2 \pi p)=0$ if $p\notin K$,  we can  write   $$I_2^h (\epsilon)=
\int_{\Tt}\int_{\Tt}(1 -\eta^\epsilon (y) )\;\zeta(x,y)\; e^{i[P\cdot \frac{y}{h}+\frac{v_0(x+\frac{y}2)  -
v_0(x-\frac{y}2)}{ h}  ]} \cdot $$
$$ \cdot \sum_{\hat{p}\in h \Zz\cap K,\, <}
 f(x,2\pi\hat{p})\,e^{\frac{2 \pi \, i\,  \hat{p}\, y}{h
}} dx dy ,$$
and using the above calculation { for  $g=f$ (note
that derivatives of $f$ with respect to $\hat{p}$ are well
controlled)}, we have that



$$|I_2^h (\epsilon)|\leq O(h^2)\sum_{\hat{p}\in h \Zz\cap K,\, <}
\int_{\Tt}\int_{\Tt}\left |\frac{1 -\eta^\epsilon (y) }{e^{2 \pi \, i\,  \, y} -1}
\;\zeta(x,y)\; e^{\frac{i}h S_{\hat p}(x,y)}\right | dx dy \sim$$

$$\sim O(h)\int_{K}
\int_{\Tt}\int_{\Tt} dx dy dp $$

since $\displaystyle \left|\frac{1 -\eta^\epsilon (y) }{e^{2 \pi \, i\,  \, y} -1}
\;\zeta(x,y) e^{\frac{i}h S_{\hat p}(x,y)}\right|\leq C,$ for some constant $C$.

\cqd

We now need to estimate $I_1^h$, for which  we need to consider two cases Ia) and Ib) depending on whether the support of $f$ intersects or not the energy level $\overline H(P)$.
\bigskip

Case Ia) Suppose that the support of $f$ does intersect the energy level $\overline H(P)$.

For each $\hat p$ we define a function

$$I_{\epsilon}^h(\hat p):=\int_{\Tt}\int_{\Tt}\, \eta^\epsilon (y)\, f(x,2\pi\hat{p})\, \zeta(x,y)\, e^{\frac{i}h S_{\hat p}(x,y)}dx dy,$$
hence $\displaystyle  I_1^h (\epsilon)=\sum_{\hat{p}\in h \Zz\cap K, \,<}I_{\epsilon}^h(\hat p)$,\, where $K$ is the $p$-projection of the support of $f$.

By the hypothesis of case I we are considering only points $p_{min}^P< \hat{p}< p_{max}^P$. In order to use  the stationary phase method to estimate $I_{\epsilon}^h(\hat p)$, we need to analyze the points $(x,y)$ such that $\nabla S_{\hat p}(x,y)=0$, it was shown that there are 4 points satisfying this, they are $(x_i(\hat p),0)$ and $(0,2x_i(\hat p))$, for $i=1,2$.
We need to calculate $D^2 S (x,y)$ for these 4 points. For $(x_i(\hat p),0)$ we have that
 $$ D^2 S (x_i(\hat p),0)= \left(
\begin{array}{cc}
0 & v_0 '' (x_i(\hat p)) \\
v_0''(x_i(\hat p))  & 0
\end{array}\right),
$$
as $2\pi \hat p\neq p^+(0)$, we have that $x_i(\hat p)\neq 0$. Hence $v_0''(x_i(\hat p))\neq 0$. Also $ \sqrt{| \text{det}\, ( D^2 S(x_i(\hat p),0))|}=|v''_0(x_i(\hat p))|\neq 0 .$

For $(0,2x_i(\hat p))$ we obtain

 $$ D^2 S (0,2x_i(\hat p))= \left(
\begin{array}{cc}
 v_0 '' (2x_i(\hat p))& 0 \\
0& v_0''(2x_i(\hat p))
\end{array}\right),
$$
$i=1,2$.

To have  $v_0''(2x_i(\hat p))\neq 0$ we need that $2x_i(\hat p)\neq 0$, this happens because  $2\pi \hat p\neq p^+(0)$, and that $y_i ( \hat{p})=2x_i(\hat p)\neq \pm 1$, but these points are avoided because $\eta^{\epsilon}(1)=\eta^{\epsilon}(-1)=0$.
Hence  $  \sqrt{| \text{det}\, ( D^2 S(0,2x_i(\hat p)))|}=|v''_0(2x_i(\hat p))|\neq 0 .$

Now we  use  the
stationary phase method, and note that $S_{\hat p}(x_i(\hat p),0)=0$, to obtain that
$$I_{\epsilon}^h(\hat p)=\sum_{i=1}^2\frac{h 2\pi}{|v''_0(x_i(\hat p))|}\, f(x_i(\hat p),2\pi\hat{p}) \;\zeta(x_i(\hat p),0) \; [1+O(\sqrt{h})]  + $$
 $$+\sum_{i=1}^2\frac{h 2\pi}{|v''_0(2x_i(\hat p))|}\eta^\epsilon (2x_i(\hat p))\, f(0,2\pi\hat{p})\;
 \zeta(0,2x_i(\hat p)) e^{\frac{i}h S_{\hat p}(0,2x_i(\hat p))}  [1+O(\sqrt{h})] .$$



Therefore
$$I_1^h (\epsilon)=J_1^h+J_2^h(\epsilon),$$ where
$$J_1^h=: \sum_{\hat{p}\in ( h \Zz+ \frac{P}{2\pi})\cap K, \, <}\,\,\sum_{i=1}^2\frac{2\pi h}{|v''_0(x_i(\hat p))|}\left[f(x_i(\hat p),2\pi\hat{p})\; \zeta(x_i(\hat p),0) \; [1+O(\sqrt{h})]\right] $$ and
$$J_2^h(\epsilon)=:\sum_{\hat{p}\in ( h \Zz+ \frac{P}{2\pi})\cap K, \, <}\,\,\,\sum_{i=1}^2\frac{2\pi h}{|v''_0(2x_i(\hat p))|}\,\cdot$$
$$\cdot\left[ \eta^\epsilon (2x_i(\hat p))\, f(0,2\pi\hat{p})\;
 \zeta(0,2x_i(\hat p)) e^{\frac{i}h S_{\hat p}(0,2x_i(\hat p))}  [1+O(\sqrt{h})]\right].$$

 Remember that $p^+(x_i (\hat{p}))=2\pi \hat p,\;\; i=1,2$  (see Remark just after (\ref{a4})). For each $p\in\mathbb{R}$ we define $x_i( p)$, $i=1,2$ be the points satisfying
\begin{equation}p^+(x_i( p))=2\pi p .\end{equation}
Therefore, when $h\to 0$ we have the following estimate
$$J_1^h (\epsilon)\sim 2\pi \sum^{2}_{i=1} \int_K\frac{f(x_i( p),2\pi{p})\; \zeta(x_i( p),0)}{|v''_0(x_i( p))|}dp$$
and
  $$J_2^h (\epsilon)\sim 2\pi \sum^{2}_{i=1} \int_K\frac{\eta^\epsilon (2x_i( p))f(0,2\pi{p})\; \zeta(0,2x_i( p))\;e^{\frac{i}h S_{ p}(0,2x_i( p))}}{|v''_0(2x_i( p))|}dp.$$

\bigskip

The trouble here is that  $v_0 '' (2x_i( {p})  )$   is  small if $x_i(p)$ is near to 0. Let us estimate  $v_0 '' (2x_i( {p})  )$.

 We have $v_0 '' (2x_i( {p})  )= v_0'''(0)\, 2x_i( {p}) \,+ O(x_i( {p})^2) $, we know that

 \noindent $p^+(x_i( {p}))=\sqrt{2[\Hh(P)-V(x_i( {p}))]}$ and
$p_{max}=\sqrt{2[\Hh(P)-V(0)]}$.

 Therefore $$p_{max}^2-(2\pi  p)^2=V(x_i( {p}))-V(0)=x_i( {p})^2 V''(0)+O(x_i( {p})^3),$$
this implies
$$ C x_i( {p})^2= 2 p_{max} (p_{max}-2\pi  p) -(p_{max}-2\pi  p)^2  +O(x_i( {p})^3),    $$
then
$$x_i( {p})^2(C+O(x_i( {p}))=(p_{max}-2\pi  p)(D-(p_{max}-2\pi  p)),  $$
hence $$x_i( {p})^2=(p_{max}-2\pi  p)\left[\frac{D-(p_{max}-2\pi  p)}{C+O(x_i( {p}))} \right] . $$
For $\epsilon>0 $  small enough, if $ |x_i(p)|\leq \epsilon$ we have that $|p_{max}-2\pi  p|\leq \epsilon^2 C$.
And
$$|v_0 '' (2 x_i( {p})  )|\approx  \tilde C \sqrt{p_{max}-2\pi  p} \,+ O({p_{max}-2\pi  p})=$$
$$= \sqrt{p_{max}-2\pi  p}\;(\tilde C+O(\sqrt{{p_{max}-2\pi  p}}\;)\approx \tilde C\sqrt{p_{max}-2\pi  p}\;.$$

\bigskip


 Now we estimate $J_2^h (\epsilon)$
$$ |J_2^h (\epsilon)|\leq 2\pi h\sum_{\hat{p}\in h \Zz\cap K, \, <}\sum_{i=1}^2\frac{M}{|v''_0(2x_i(\hat p))|} \eta^\epsilon (2x_i(\hat p))\,\sim$$
 $$\sim 2\pi\sum_{i=1}^2 \int_K \eta^\epsilon (2x_i( p))\frac{\tilde M}{\sqrt{p_{max}-2\pi  p}} \; dp\leq\tilde M 2\pi\sum_{i=1}^2 \int_{C(\epsilon)} \frac{1}{\sqrt{p_{max}-2\pi  p}}\,dp, $$ where $C(\epsilon)=\{p\;|\, 0< p_{max}-2\pi p\leq C\epsilon^2\}$.
The contribution over the variable $p$ is like $\int_{0}^{C\epsilon^2}
\frac{1}{\sqrt{p} }\, dp$, which  gives a factor of
$\sqrt{C\epsilon^2}$. { We point out that in the Riemann  sum above (which does not contain the contribution of $ \hat{p}_d$) the error of approximation for the integral  $\int_{0}^{C\epsilon^2}
\frac{1}{\sqrt{p} }\, dp$ is of order $h$. }

Finally, for each $h$ fixed we  consider $\hat{p}_{h}=\hat{p}_{d(h)}$ as before, and we analyze    the term
$$\int_{\Tt}\int_{\Tt}\,  f(x,2\pi\hat{p}_{h})\; \zeta(x,y) \;e^{\frac{i}{h}S_{\hat p_{h}}(x,y)}\;dx dy.$$
Let $\delta>0 $ be fixed, we define the set $B_{\delta}=\{(x,y) ; |x|\leq {\delta}, y\in [-\frac{1}2, \frac{1}2] \}\cup \{(x,y) ; |y|\leq {\delta}, x\in [-\frac{1}2, \frac{1}2] \}.$ Then, as the stationary points are $(x_i (\hat{p}),0)$ and  $(0,2x_i (\hat{p}))$, for $i=1,2$, we see that this points are in $B_{\delta}$. Hence $|\nabla S_{\hat p_{h}}(x,y) |>K({\delta})  $ in $B_{\delta}^c$, where $K({\delta})$ is uniformly bounded by below as $h\to 0$, for fixed $\delta$.
Therefore we have

$$\left|\int_{\Tt}\int_{\Tt}\,  f(x,2\pi\hat{p}_{h})\; \zeta(x,y) \;e^{\frac{i}{h}S_{\hat p_{h}}(x,y)}\;dx dy\right|\leq$$  $$\leq\left|\int_{B_{{\delta}}}\,  f(x,2\pi\hat{p}_{h})\; \zeta(x,y) \;e^{\frac{i}{h}S_{\hat p_{h}}(x,y)}\;dx dy\right| + $$

$$+\left|\int_{{B_{{\delta}}^c}}\,  f(x,2\pi\hat{p}_{h})\; \zeta(x,y) \;e^{\frac{i}{h}S_{\hat p_{h}}(x,y)}\;dx dy\right|$$

Note that $$\left|\int_{B_{{\delta}}}\,  f(x,2\pi\hat{p}_{h})\; \zeta(x,y) \;e^{\frac{i}{h}S_{\hat p_{h}}(x,y)}\;dx dy\right| \leq C \delta$$

And $$\int_{{B_{{\delta}}^c}}\,  f(x,2\pi\hat{p}_{h})\; \zeta(x,y) \;e^{\frac{i}{h}S_{\hat p_{h}}(x,y)}\;dx dy=$$

$$\frac{h}{i}\int_{\partial ({B_{{\delta}}^c})} e^{\frac{i}h S_{\hat p}(x,y)}\frac{f(x,2\pi\hat{p})\; \zeta(x,y)\,n(x,y)^{T}\nabla S_{\hat p}(x,y)}{|\nabla S_{\hat p}(x,y)|^2}\,
   dS-$$
$$- \frac{h}{i}\int_{ {B_{{\delta}}^c}} e^{\frac{i}h S_{\hat p}(x,y)}\,\nabla\left[\frac{f(x,2\pi\hat{p})\; \zeta(x,y)\,\nabla S_{\hat p}(x,y)}{|\nabla S_{\hat p}(x,y)|^2}\,
  \right] dxdy    $$

  Hence $$\left| \int_{{B_{{\delta}}^c}}\,  f(x,2\pi\hat{p}_{h})\; \zeta(x,y) \;e^{\frac{i}{h}S_{\hat p_{h}}(x,y)}\;dx dy\right| \leq h \bar K({{\delta}}).$$
Thus sending $h\to 0$, and then $\delta \to 0$ we see that this term vanishes.

\bigskip
Now we are able to estimate $I_f(h)$.

 Remember that $\zeta(x,y)= e^{[-v_1(x+\frac{y}2)-v_1(x-\frac{y}2)+2v_1(\bar x)] }$. We will show in the last section that $v_1(x)=\frac{1}{2}\ln (p^+(x))$, hence $\zeta(x_i( p),0)=\frac{p^+(\bar x)}{p^+(x_i( p))} $ and by the definition of $v_0$, $v''_0(x_i( p))=(p^{+})'(x_i( p))$. Therefore
$$J_1^h (\epsilon)\sim 2\pi \sum^{2}_{i=1} \int\,\frac{f(x_i( p),2\pi p)\,p^+( \bar x)} {|(p^{+})'(x_i( p))|\,p^+(x_i( p))}\,dp .$$

Above we are approximating an integral by a Riemann sum.

Hence, as $\epsilon,\delta>0$ are arbitrary we get that
$$I_f(h)\to 2\pi \sum^{2}_{i=1} \int\frac{f(x_i( p),2\pi{p})\;p^+(\bar x)} {|(p^{+})'(x_i( p))|\;p^+(x_i( p))}\,dp\; .$$

It is well known, via action angles variables  (see  \cite{PR} \cite{Ar}), that the projected Mather measure is
given by the density
$$ b_P(x)= b(x) = \frac{\partial^2 G_P}{\partial x\, \partial P}= \,\frac{\frac{\partial \overline{H}(P)}{\partial P} }{ p^+_P (x) }, $$
were $G_P(x)=\int_{-\frac{1}2}^x \, p^+_P (s)ds  $.

Using the change of coordinates $x_i$ in each branch we get that for any given function
$g(x)$ we have
$$ \int g(x) \,b(x) dx = \int  \,
\frac{g( x_i (p)) \,b(x_i (p))}{ (p^+ )\,'(x_i (p)) }\, dp \; .$$

This describes the limit measure in this case.

\medskip

We postpone the analysis of case Ib) as it will use certain results which will be established in case II).

\medskip

Case II) Suppose that the support of $f$ is contained in $\Tt\times [c, d]$, with $c,d$  such that
$[c, d] \cap [p^P_{min}, p^P_{max}]=\emptyset$. Hence, by  item (ii) above, for each $\hat p\in h\Zz+ \frac{P}{2\pi}$ such that $2\pi\hat p\in [c,d]$ we have $\nabla S_{\hat p}(x,y)\neq 0$ for all $(x,y)\in \Tt\times \Tt$. Also, we note that $|\nabla S_{\hat p}(x,y)|^2$ is uniformly bounded away from 0 in all variables.

We  write
 $$I_f(h)=\sum_{\hat{p}\in h \Zz+ \frac{P}{2\pi}}I_h(\hat p),  $$
where $$ I_h(\hat p)=\int_{\Tt}\int_{\Tt}\,  f(x,2\pi\hat{p})\; \zeta(x,y) \;e^{\frac{i}{h}S_{\hat p}(x,y)}\;dx dy.$$

Using  the   stationary phase method when $\Tt\times \Tt$ contains no stationary  points (see \cite{Mi} sec. 5.7,  or Theorem 7.7.1 \cite{Ho}), we get
$$I_h(\hat p)= \frac{h}{i}\int_{\partial (\Tt\times\Tt)} e^{\frac{i}h S_{\hat p}(x,y)}\frac{f(x,2\pi\hat{p})\; \zeta(x,y)\,n(x,y)^{T}\nabla S_{\hat p}(x,y)}{|\nabla S_{\hat p}(x,y)|^2}\,
   dS-$$
$$- \frac{h}{i}\int_{ \Tt\times\Tt} e^{\frac{i}h S_{\hat p}(x,y)}\,\nabla\left[\frac{f(x,2\pi\hat{p})\; \zeta(x,y)\,\nabla S_{\hat p}(x,y)}{|\nabla S_{\hat p}(x,y)|^2}\,
  \right] dxdy   =:F+G, $$
    where $n$ is an exterior normal unit vector and $dS$ denotes an element of surface area on the boundary $\partial (\Tt\times\Tt)$.

Note that the integrand is not $1$-periodic in $y$ so the boundary term does not vanish.

To analyze $F$, note that $$F=\frac{h}{i}\int_{\partial (\Tt\times\Tt)} e^{\frac{i}h S_{\hat p}(x,y)}\frac{f(x,2\pi\hat{p})\;
\zeta(x,y)\,n(x,y)^{T}\nabla S_{\hat p}(x,y)}{|\nabla S_{\hat p}(x,y)|^2}\,
   dS$$

$$=-\frac{h}{i}\int_{-\frac{1}2}^{\frac{1}2} e^{\frac{i}h S_{\hat p}(x,-\frac{1}2)}\frac{f(x,2\pi\hat{p})\;
\zeta(x,-\frac{1}2)\,\frac{\partial S_{\hat p}}{\partial y}(x,-\frac{1}2)}{|\nabla S_{\hat p}(x,-\frac{1}2)|^2}\,
   dx+$$

$$+\frac{h}{i}\int_{-\frac{1}2}^{\frac{1}2} e^{\frac{i}h S_{\hat p}(\frac{1}2,y)}\frac{f(\frac{1}2,2\pi\hat{p})\;
\zeta(\frac{1}2,y)\,\frac{\partial S_{\hat p}}{\partial x}(\frac{1}2,y)}{|\nabla S_{\hat p}(\frac{1}2,y)|^2}\,
   dy+$$

  $$+\frac{h}{i}\int_{-\frac{1}2}^{\frac{1}2} e^{\frac{i}h S_{\hat p}(x,\frac{1}2)}\frac{f(x,2\pi\hat{p})\;
\zeta(x,\frac{1}2)\,\frac{\partial S_{\hat p}}{\partial y}(x,\frac{1}2)}{|\nabla S_{\hat p}(x,\frac{1}2)|^2}\,
   dx-$$

   $$-\frac{h}{i}\int_{-\frac{1}2}^{\frac{1}2} e^{\frac{i}h S_{\hat p}(-\frac{1}2,y)}\frac{f(-\frac{1}2,2\pi\hat{p})\;
\zeta(-\frac{1}2,y)\,\frac{\partial S_{\hat p}}{\partial x}(-\frac{1}2,y)}{|\nabla S_{\hat p}(-\frac{1}2,y)|^2}\,
   dy.$$

As the functions $S_{\hat p}$, $f$ and  $\zeta$ are periodic in the variable $x$,  the sum of second and the fourth term on the right hand side
of the equality above is zero. Also, by the periodicity in $x$ of the functions involved, the first and the third terms are equal to 0.

Hence $F=0$.

Let us now analyze $G$, define
$$\displaystyle g_{\hat p} (x,y):=\nabla\bigg[\frac{f(x,2\pi\hat{p})\; \zeta(x,y)\,\nabla S_{\hat p}(x,y)}{|\nabla S_{\hat p}(x,y)|^2}\bigg],\,
 $$
we can repeat the calculation and get
 $$ G=-\left(\frac{h}{i}\right)^2 \int_{\partial (\Tt\times\Tt)} e^{\frac{i}h S_{\hat p}(x,y)}\frac{ g_{\hat p}(x,y)\,n(x,y)^{T}\nabla S_{\hat p}(x,y)}{|\nabla S_{\hat p}(x,y)|^2}\,
   dS +$$
   $$+\left(\frac{h}{i}\right)^2\int_{ \Tt\times\Tt} e^{\frac{i}h S_{\hat p}(x,y)}\,\nabla\left[\frac{g_{\hat p}(x,y)\,\nabla S_{\hat p}(x,y)}{|\nabla S_{\hat p}(x,y)|^2}\,
  \right] dxdy.
  $$

Then, as $\displaystyle I_f(h)=\sum_{\hat{p}\in h \Zz} F+G$, we have
$$I_f(h)={h}^2 \sum_{\hat{p}\in h \Zz+ \frac{P}{2\pi}}\int_{\partial (\Tt\times\Tt)} e^{\frac{i}h S_{\hat p}(x,y)}\frac{ g_{\hat p}(x,y)\,n(x,y)^{T}\nabla S_{\hat p}(x,y)}{|\nabla S_{\hat p}(x,y)|^2}\,
   dS  - $$
   $$ -h^2\sum_{\hat{p}\in h \Zz+ \frac{P}{2\pi}}\int_{ \Tt\times\Tt} e^{\frac{i}h S_{\hat p}(x,y)}\,\nabla\left[\frac{g_{\hat p}(x,y)\,\nabla S_{\hat p}(x,y)}{|\nabla S_{\hat p}(x,y)|^2}\,
  \right] dxdy =:G_1+G_2.$$

We will  show that $|G_1|=O(h)$, by similar arguments we can show that  $|G_2|=O(h)$. As $|e^{\frac{i}h S_{\hat p}(x,y)}|=1$, we have
  $$|G_1|\leq {h} \left|h\sum_{\hat{p}\in h \Zz+ \frac{P}{2\pi}}\int_{\partial (\Tt\times\Tt)} \frac{ g_{\hat p}(x,y)\,n(x,y)^{T}\nabla S_{\hat p}(x,y)}{|\nabla S_{\hat p}(x,y)|^2}\,
   dS \right|\sim $$
   $$\sim {h} \left|\int\int_{\partial (\Tt\times\Tt)} \frac{ g_{ p}(x,y)\,n(x,y)^{T}\nabla S_{ p}(x,y)}{|\nabla S_{ p}(x,y)|^2}\,
   dS dp \right| =O(h), $$ the last equality is true  because $p\in [c,d]$, all functions involved are sufficiently smooth and $|\nabla S_{\hat p}(x,y)|^2$ is uniformly bounded away from 0.

\medskip
Case Ib) Suppose that the support of $f$ does not intersect the energy level $\overline H(P)$.

 We have already shown that $I_2^h  (\epsilon)=O(h)$, hence we just need to analyze $$I_{1}^h(\epsilon)=\sum_{\hat{p}\in h \Zz+ \frac{P}{2\pi}}\int_{\Tt}\int_{\Tt}\, (\eta^\epsilon (y) )\, f(x,2\pi\hat{p})\, \zeta(x,y)\, e^{\frac{i}h S_{\hat p}(x,y)}dx dy.$$

We will suppose that $\supp (f)\subset [a,b]\times[c,d]=A$, where $A$ is such that $A\cap \{(x,p^+(x))\,|\, x\in[-\frac{1}2,\frac{1}2]\} =\emptyset$.

For each $\hat p$ fixed the  critical points $(x_i(\hat p),0)$ for $S_{\hat p}(x,y)$  do not contribute in the integral, because $f(x_i(\hat p), 2\pi \hat p)=0$. Then or we have the contribution of the points $(0,2x_i(\hat p))$, that we have already shown that is of order $O( {\epsilon})$, or in $A$ there are not stationary points. In this last case,  there exists $K_{\hat p}>0$ such that $|\nabla S_{\hat p}(x,y)|^2>K_{\hat p}$ if $f(x, 2\pi \hat p)\neq 0$. As $\hat p$ lies in a compact set there exists $K>0$, uniformly in $\hat p$, such that $|\nabla S_{\hat p}(x,y)|^2>K$ if $f(x, 2\pi \hat p)\neq 0$.

Also, note that the function $\eta^\epsilon (y)$ is smooth and periodic in $y$. Therefore we can apply the same argument used in case II) to prove that
$|I_1^h(\epsilon)|=O(h)$.

\bigskip
\Pf(of Theorem \ref{teo1}) By the analysis above we see that the limit measure is supported in the graph $\{(x,p^+(x))\,|\, x\in[-\frac{1}2,\frac{1}2]\} $.

Item a) follows by  case Ia),
item b) follows by case Ib) and case II) above.
\cqd

\section{ The asymptotic expansion  of $v_h$ }

\bigskip

Let $v_h$ and $v_h^*$ be $\Tt$-periodic solutions of
$$ - \, \frac{h\,  v_h''}{2}+ \frac{1}{2} \, |\,P +  v_h' \,|^2 + V= \overline{H}_h(P),$$
and,
$$  \, \frac{h\,  {v_h^*}''}{2}+ \frac{1}{2} \, |\,P +  {v_h^*}' \,|^2 + V= \overline{H}_h(P).$$

We will provide rigorous asymptotic  expansions for  $v_h$ and $v_h^*$, as required in the previous sections. We will only present the proofs for $v_h$ as the analysis of $v_h^*$ is completely similar.

For a given $N$ we will construct functions $v_i$, numbers $\overline{H}_i$, $i=0,...,N$, and formal expansions
 $$\hat{v}^N_h= v_0  + h\, v_1 +  h^2\, v_2+... + h^N\,
v_N,$$  and  $$\hat{H}_h^N=  \overline{H}_0 + h  \overline{H}_1 + h^2 \overline{H}_2 +...+ h^N  \overline{H}_N.$$
We will show that
$$\overline H_h=\overline{H}_h(P) =  \hat{H}_h^N+ O(h^{N+1}),$$
and
$$ {v}_h(x)-{v}_h(x_h)=\hat{v}^N_h(x)-\hat{v}^N_h(x_h)+
O(h^{\frac{N+1}2}).  $$

In order to get these estimates  we will derive differential equations which define
$v_i$ and $H_i$ in a suitable way. In fact, $v_0$
and $H_0$ are solutions of the limit problem
$$ \overline{H}_0 =H(v_{0}'(x),x)= V(x) + \frac{1}2 \, ( P + v_{0}'(x) )^2. \,\,\,\,*(1)$$

The remaining equations are obtained by Taylor expansions in formal
powers of $h$. For instance, the following equations will be
$$ \overline{H}_1 =\,- \frac{v_{0}''}2 +  H_p (v_{0}',x)\,v_{1}'=\,- \frac{v_{0}''}2 +
( P + v_{0}' )\,v_{1}',\,\,\,\,*(2)
$$
and
$$ \overline{H}_2 =\,- \frac{v_{1}''}2 +  H_p (v_{0}',x)\,v_{2}'+ \frac{( v_{1}')^2}{2}\, .\,\,\,\,*(3)$$

 All approximations in this section are $O(h^{\frac{N+1}2})$ in the uniform topology.

From $*(1)$ we have that $v_0^{}=\phi$ and $\overline{H}_0=\overline
H(P)$, where $\phi(x) = \int_{-\frac{1}{2}}^x \, p^+_P (s)ds -P\,(
x+\frac{1}{2}).$ Using this last expression in $*(2)$, we have an
equation for $v_1^{}$ and $\overline{H}_1$. The value
$\overline{H}_1$ is obtained by the compatibility condition:
$$ \overline{H}_1 = - \, \int v_{0}''\,d \sigma_0,
$$
where $\sigma_0$ is the projected Mather probability for $h=0$,
which satisfies div $(( P + v_{0}') \, \sigma_0) =0$.

 The property
div $(( P + v_{0}') \, \sigma_0) =0$ follows from the fact that the Mather measure $\mu$ has support on the graph $ (x, P+v_0�(x))$, and also satisfies the holonomic condition
$\int v D_x\varphi (x) d\mu=0$. This implies for the projected measure $\sigma_0$ that
$\int (P+v_0�(x)) \varphi_x d\sigma_0=0$ (see \cite{Ev1}, \cite{BG} or \cite{EG1}).

Then, once determined $\overline{H}_1$, the function $v_1$ is uniquely
determined (up to additive constant) because $( P + v_{0}')$ is
never zero.

Remember that $v_{0}''=(p^+_P)'$ and that the  density of the projected Mather measure is given by
$ b_P(x)= \,\frac{\frac{\partial \overline{H}(P)}{\partial P} }{ p^+_P (x) } .$ Hence

$$ \overline{H}_1 = - \, \frac{\partial \overline{H}(P)}{\partial P}\int_{-\frac{1}2}^{\frac{1}2}  \,\frac{(p^+_P(x))' }{ p^+_P (x) }dx=- \, \frac{\partial \overline{H}(P)}{\partial P} \ln (p^+_P(x))|_{-\frac{1}2}^{\frac{1}2} =0,
$$
and the equation $*(2)$ becomes
$$  0 =\,- \frac{(p^+_P)'}2 +p^+_P \,v_{1}'.
$$
Therefore
$$v_1(x)= \frac{1}2\ln (p^+_P(x)).$$

Using the function $v_1$ in $*(3)$  we have an equation for $v_2$
and $\overline{H}_2$.
There exists a unique number $\overline{H}_2$, given by the compatibility condition
$$ \overline{H}_2\,= \,\int \, [ - \frac{v_{1}''}{2}+ \frac{(
v_{1}')^2}{2}  ] d\sigma_0,$$ to which corresponds a unique (up to
constant) $v_2$. Again, the solvability is ensured by the condition
$( P + v_{0}')\neq 0$.

In this way we obtain, inductively, all the $v_j^{}$, and all the
$\overline{H}_j$, $j=0,1,...,N$. Note that the functions  $v_j$ are
smooth, because $V$ is smooth and consequently $p^+$ also is. 

\begin{pro}
For all $N$ and for all $x$
\begin{equation}
\label{aexp}
 - \frac{h (\hat{v}_{h}^{N})''(x)}{2}  + \frac{1}2 |P+(\hat{v}_{h}^{N})'(x)|^2+V(x) =\hat{H}_h^N+ O( h^{N+1})
.
\end{equation}
\end{pro}

\Rm As the proof below shows, because the expansion $\hat{v}_{h}^{N}$ is a finite sum of smooth functions, the error term in the right
hand-side of \eqref{aexp} is $O(h^{N+1})$ in the $C^k$ topology for all $k\geq 0$.

\Pf
We will show the claim for $N=2$, as all other cases are analogous.
We want to show that
$$  - \frac{h (\hat{v}_{h}^{2})''(x)}2  + H( (\hat{v}_{h}^{2})'(x),x)- \hat H_h^2= O(h^3) .$$

In order to do so, we expand, for each  fixed $x$, $H( (\hat{v}_{h}^{2})'(x),x)$ in Taylor series:
$$ H( (\hat{v}_{h}^{2})'(x),x) =$$
$$=H(v_{0}',x)+hH_p(v_{0}',x)(v_{1}'+hv_{2}')+
\frac{h^2}2H_{pp}(v_{0}',x)(v_{1}'+hv_{2}')^2+O(h^3).   $$
Hence, by definition of $v_i,\overline{H}_i$, $i=0,1,2$ we have
$$ - \frac{h (\hat{v}_{h}^{2})''(x)}2  + H( (\hat{v}_{h}^{2})'(x),x)- \hat H_h^2=
$$
$$=-\frac{h}2(v_{0}''+hv_{1}''+h^2v_{2}'')+H(v_{0}',x)+hH_p(v_{0}',x)(v_{1}'+hv_{2}')+
$$
$$+\frac{h^2}2H_{pp}(v_{0}',x)(v_{1}'+hv_{2}')^2+O(h^3)-(\overline{H}_0+h\overline{H}_1+h^2 \overline{H}_2)=$$
$$=h^3[-v_{2}''+ H_{pp}(v_{0}',x)(2v_{1}'v_{2}'+ h(v_{2}')^2)]+O(h^3) = O(h^3). $$

 \cqd

\begin{pro}
$|\overline{H}_h - \hat{H}^N_h|= O(h^{N+1})$
\end{pro}

\Pf

From \cite{Gom1}, section 3, we have that
$$\overline{H}_h (P) \leq
\sup_{x} \,\{\, - \frac{h \,(\hat{v}_{h}^{N})''(x)}2 + H( (\hat{v}_{h}^{N})'(x)
,x)\,\}\leq  \, \hat{H}^N_h + O(h^{N+1}).$$
Therefore, $\overline H_h - \hat{H}^N_h \leq O(h^{N+1}).$

As $H(p,x)$ is convex on $p$, we have that
$$\overline H_h - \hat{H}^N_h\geq $$
$$H( v_{h}' (x),x) - H((\hat{v}_{h}^N)' (x),x  ) -\frac{h}2\, (   v_{h}'' (x) - (\hat{v}_{h}^N)'' (x)    )+ O(h^{N+1})\geq  $$
$$ H_p ((\hat{v}_{h}^N)' (x) , x)\,(\,v_{h}' (x)- (\hat{v}_{h}^N)' (x)\, ) \, - \frac{h}2\, (   v_{h}'' (x) - (\hat{v}_{h}^N)'' (x)    )+
O(h^{N+1}).$$

Consider a point $x_0$ where $(   v_{h} (x) - \hat{v}_{h}^N (x) )$
has a maximum. Then, $ (   v_{h}' (x_0) - (\hat{v}_{h}^N)' (x_0)    )=0$
and $h\, ( v_{h}'' (x_0) - (\hat{v}_{h}^N)'' (x_0)    )\leq 0$.

From the above $\overline H_h - \hat{H}^N_h   \geq O(h^{N+1}).$
\cqd

\begin{lem}

There exists a periodic non negative probability density $\sigma_{h}^{N}$ that is a solution
 of the equation
$$ (g^N_h (x)\, \sigma_{h}^{N} (x))' + h (\sigma_{h}^{N})''=0,$$
where $g^N_h(x)=H_p((\hat{v}_{h}^N)'(x) ,x)$. Also, there exist constants $k,K>0$ such that $k\leq\sigma_{h}^{N}(x) \leq K $.
\end{lem}

\Pf
Note that $g^N_h=P+v_{0}'+h v_{1}'+h^2v_{2}'   $ is positive and bounded uniformly  in $h$.
For convenience we will drop the $N$.
By integration, we get that the equality for all $x$
$$ (g_h(x) \sigma_h (x))' + h \sigma_h{}''=0,$$
is equivalent to
$$ g_h(x)\, \sigma_h (x) + h \sigma_h'(x)=c_h,$$
for some suitable constant $c_h$.

We multiply the above expression in each side by the integrating
factor $\alpha_h(x)  = e^{ \frac{G(x)}h}$, where $G(x)= \int_{0}^x g_h(s)
ds.$

In this way,
$$ h\, ( \alpha_h \, \sigma_h)'=h \alpha_h'(x) \sigma_h(x)  +    h \alpha_h(x) \sigma_h'(x)=
$$
$$\alpha_h (x)\,g_h(x)\, \sigma_h (x) + h \alpha_h(x) \sigma_h'(x)=c_h
\alpha_h(x).$$

Therefore,
$$ h\, \alpha_h(x) \sigma_h (x)                  \,=\,c_h \int_0^x \alpha_h (s)\, ds +
C_h,$$
or, equivalently,
$$\sigma_h(x) =\frac{  c_h \int_0^x  e^{ \frac{G(s)}h}(s)\, ds +
C_h   }{e^{ \frac{G(x)}h}  \,h}.$$

The constants $c_h$ and $C_h$ are adjusted so that  $\sigma_h$ is a periodic probability density.

We claim that the constants $c_h$ and $C_h$ are bounded. Note first
that
$$\int _0^1 g_h(x)\, \sigma_h (x)\,dx= \int _0^1 [g(x)\, \sigma_h (x) + h \sigma_h'(x)]\, d x=c_h.$$

As $g_h$ is positive and bounded uniformly on $h$,  and
$\sigma_h$ is a density, we get that  $c_h$ is bounded. In fact, $c_h$
converges to a constant, that we will denote by $c$.

By Laplace method, see \cite{Ol} (chapter 3, section 7)  \cite{Mi}, as $G$ is monotonous increasing, we have
$$  \frac{   \int_0^x e^{ \frac{G(s)}h} \, ds}{e^{\frac{ G(x)}h}
\,h}\sim \frac{1}{g_h(x)}.  $$
  Furthermore, the limit is uniformly
bounded on $h$. Now we will show that $C_h$ is uniformly bounded in
$h$.

As
$$\sigma_h (0) = \sigma_h(1),  $$
we get
$$\frac{C_h}{ h}=\frac{C_h}{e^{ \frac{G(0)}h}  h}=\frac{C_h   + c_h \int_0^1 e^{ \frac{G(s)}h} \, ds}{e^{\frac{ G(1)}h}  h}.$$
Consequently,
$$ C_h = C_h \, e^{ - \frac{G(1)}h}  + \frac{c_h \int_0^1 e^{ \frac{G(s)}h} \, ds}{e^{ \frac{G(1)}h}  h}\, h .$$
From this
$$ C_h (1-\, e^{ - \frac{G(1)}h})  = \frac{c_h \int_0^1 e^{ \frac{G(s)}h} \, ds}{e^{\frac{ G(1)}h}  h}\, h .$$
In this way, $C_h \to 0$, as $h\to 0$. Therefore for any $x$
$$\lim_{h\to 0} \sigma_h(x) =\lim_{h\to 0} \frac{  c_h \int_0^x \alpha_h (s)\, ds + C_h   }{e^{ \frac{G(x)}h}
\,h}= \lim_{h\to 0} \frac{  c_h \int_0^x  e^{\frac{ G(s)}h} \, ds }{e^{
\frac{G(x)}h} \,h}$$
$$=\frac{c}{P+v_{0}'(x)}>0$$
This shows that $\sigma_{h}$ is bounded and bounded away from zero.
\cqd

\begin{pro}\label{estimL2}
\begin{equation}
\label{aexp2}
|   v_{h}'  - (\hat{v}^{N}_{h})'  |_{L^2}= O(h^{\frac{(N+1)}2}).
\end{equation}
\end{pro}
\Pf

Note that from the last
proposition
$$O(h^{N+1}) = $$
$$\int \,\left[\,H( (\hat{v}_{h}^{N})'(x),x) - \frac{h \,(\hat{v}_{h}^{N})''(x)}2  \, - H( v_h'(x) ,x)) + \frac{h \,v_h''(x)}2  \,\right]\, \sigma_h^N(x) dx.$$

From the strict convexity of the Hamiltonian, and because $H(p,x)$ is quadratic in $p$
we have the following inequality
$$H( v_{h}' (x),x) - H((\hat{v}_{h}^N)' (x),x  ) -\frac{h}2\, (   v_{h}'' (x) - (\hat{v}_{h}^N)'' (x)    )\geq  $$
$$ H_p ((\hat{v}_{h}^N)' (x) , x)\,(\,v_{h}' (x)- (\hat{v}_{h}^N)' (x)\, ) \,+\, \frac{\gamma}{2}\,(\,v_{h}' (x)- (\hat{v}_{h}^N)' (x)\, )^2 -$$
$$ \frac{h}2\, (   v_{h}'' (x) - (\hat{v}_{h}^N)'' (x)    ).$$
Let $\sigma_{h}^{N}$ be as in the previous Lemma. We claim that for
any periodic function $w$ we have
\begin{equation} \label{d1} \int(H_p ((\hat{v}_{h}^{N})'(x) ,x) w'(x)-\frac{h
w''(x)}2)\sigma_{h}^{N}(x) = 0. \,\,\,\,\,\, \end{equation}

This follows from integration by part and from the last lemma, with $g_h^N(x)=H_p ((\hat{v}_{h}^{N})'(x) ,x) $, for
all $x$ we have
$$[\, (\,H_p ((\hat{v}_{h}^{N})'(x) ,x) \,\sigma_{h}^{N}(x)\,)'   -h (\sigma_{h}^{N})''(x)\, ]\,w(x) = 0 .$$

It follows from (\ref{d1}) applied to $w = v_{h} - \hat{v}_{h}^N $
that
$$ \int\,(\,v_{h}' (x)- (\hat{v}_{h}^N)' (x)\, )^2\,  \sigma_{h}^{N} (x)
dx= O(h^{N+1}).$$

As $\sigma_h^N$ is bounded and bounded away from zero we get
$$ \int\,|\,v_{h}' (x)- (\hat{v}_{h}^N)' (x)\, |^2\,
dx= O(h^{N+1}), $$ and, so
$$\|\,v_{h}' - (\hat{v}_{h}^N)' \,
\|_{L^2}= O(h^{\frac{N+1}2}). $$
 \cqd
\begin{pro}\label{estimC0}
\begin{equation}
\label{aexp3}
| {v}_h(x)-{v}_h(y)-(\hat{v}^N_h(x)-\hat{v}^N_h(y))|=O(h^{\frac{N+1}2}).
\end{equation}
\end{pro}

\Rm The expansion in \eqref{aexp3} holds only in the uniform topology, as $v_h$ is
in general only continuous. Nevertheless the function $(\hat{v}^N_h$ is smooth.

\Pf Define $g_h(x)= {v}_h(x)-\hat{v}^N_h(x)$, this function is
$\Tt$-periodic. Using Morrey's Inequality with $n=1, p=2$, we have
that
$$ |g_h(x)-g_h(y)|\leq C \|g_{h}'\|_{L_2(\mathbb{R})},$$ for all $x,y\in \Tt$. By proposition \ref{estimL2} we get the result.

\cqd

The analysis of $v_h^*$ is similar to the case we just described.
\bigskip

 We would like to thanks the referee for many valuable comments and for a very careful reading of the manuscript.

\end{document}